\pdfoutput=1
\documentclass{amsart}

\usepackage{graphicx,mathrsfs,amssymb,amsmath}
\usepackage{tikz}
\usepackage[font=small,labelfont=bf]{caption}
\usepackage[accsupp]{axessibility}

\usepackage[pdfdisplaydoctitle=true,colorlinks=true,urlcolor=blue,citecolor=blue,linkcolor=blue,
pdfstartview=FitH,pdfpagemode=None,bookmarksnumbered=true]{hyperref}

\usepackage{cmap} % make the PDF file "searchable and copyable"

\newtheorem{theorem}{Theorem}[section]
\newtheorem{lemma}[theorem]{Lemma}
\newtheorem{corollary}[theorem]{Corollary}
\newtheorem{proposition}[theorem]{Proposition}
\newtheorem{remark}[theorem]{Remark}
\newtheorem{definition}[theorem]{Definition}

\numberwithin{equation}{section}

\newcommand{\cz}{{\mathbb C}}
\newcommand{\gz}{{\mathbb Z}}
\newcommand{\nz}{{\mathbb N}}
\newcommand{\rz}{{\mathbb R}}
\newcommand{\sz}{{\mathbb S}}

\newcommand{\bfA}{\mathbf{A}}
\newcommand{\bfB}{\mathbf{B}}

\newcommand{\bfT}{\mathbf{T}}
\newcommand{\bfGa}{\mathbf{\Gamma}}
\newcommand{\wtbfG}{\mathbf{\widetilde{\Gamma}}}

\newcommand{\wtbfS}{\mathbf{\widetilde{S}}}

\newcommand{\calH}{\mathcal{H}}

\newcommand{\scrL}{\mathscr{L}}
\newcommand{\scrS}{\mathscr{S}}

\newcommand{\dbar}{d\hspace*{-0.08em}\bar{}\hspace*{0.1em}}
\newcommand{\eps}{\varepsilon}
\newcommand{\forget}[1]{}
\newcommand{\Hom}{{\mathrm{hom}}}
\newcommand{\lra}{\longrightarrow}
\newcommand{\rpbar}{\overline{\rz}_+}
\newcommand{\spk}[1]{\langle#1\rangle}

\newcommand{\wh}{\widehat}

\newcommand{\wt}{\widetilde}

\begin{document}
%%%%%%%%%%%%%%%%%%%%%%%%%%%%%%%%%%%%%%%%
%%%%%%%%%%%%%%%%%%%%%%%%%%%%%%%%%%%%%%%%
\title[Resolvent trace asymptotics for operators in the Shubin class]{
Resolvent trace asymptotics for\\ operators in the Shubin class}

\author{J\"org Seiler}
\address{Dipartimento di Matematica, Universit\`{a} di Torino, Italy}
\email{joerg.seiler@unito.it}

\begin{abstract}
A new pseudodifferential calculus of Shubin type is introduced. 
The calculus contains operators depending on a non negative real parameter as well as operators independent of the parameter. 
%Elements in the calculus intrinsically possess a Grubb-Seeley type expansion. 
Resolvents of Shubin type pseudodifferential operators are constructed  and their 
trace expansion is obtained. 

\vspace*{5mm}



%\vspace*{5mm}

%\textit{Dedicated to Professor Robert Denk in occasion of his 60th birthday.}

\forget{
\vspace*{5mm}

\noindent
Dedicated to Professor Robert Denk in occasion of his 60th birthday.
\begin{itemize}
\item[$\square$] his 70th birthday.
\item[$\square$] his liberation from the job as Editor-in-Chief for Mathematische Nachrichten.
%\item[$\square$] the completion of the last Ueblis for Mathematische Nachrichten.
\item[$\square$] his conquering of all mountain peaks up to 250m in the Himalaya region. 
\item[$\square$] his 1000th paid working day consisting in a lunch break with prior morning coffee break 
 and following afternoon coffee break. 
\item[$\square$] the publication of his 50th paper entirely written by his co-authors. 
\end{itemize}
Please tick (multiple choices permitted).
} 
\end{abstract}

\vspace*{-20mm}

\maketitle

\tableofcontents

%%%%%%%%%%%%%%%%%%%%%%%%%%%%%%%%%%%%%%%%
%%%%%%%%%%%%%%%%%%%%%%%%%%%%%%%%%%%%%%%%
\section{Introduction} 

In \cite{Shub71}, Shubin introduced a calculus of pseudodifferential operators on $\rz^n$ which permits to construct the 
parametrix of the harmonic oscillator $A=|x|^2-\Delta$, where $\Delta$ is the Laplacian on $\rz^n$. The calculus is build  
upon symbol-classes $\Gamma^d_{1,0}(\rz^{2n})$, $d\in\rz$, consisting of smooth functions $a(x,\xi)$ satisfying  
 $$|\partial^\alpha_z a(z)|\lesssim (1+|z|^2)^{(d-|\alpha|)/2},\qquad z:=(x,\xi),$$
for derivatives of arbitrary order, and corresponding subclasses $\Gamma^d(\rz^{2n})$ of poly-homogeneous symbols which, in addition, admit an asymptotic expansion into homogeneous components of (one-step) decreasing orders. Ellipticity means invertibility of the homogeneous principal symbol on the unit-sphere in $\rz^{2n}_z$; the harmonic oscillator is elliptic, since its homogeneous principal symbol is constant $1$ on the $z$ unit-sphere. 

There exists a significant amount of literature on Shubin type operators and their generalizations and applications. This includes spectral theory \cite{CGPR,Helffer,NicolaRodino,Shub78}, wave front sets and propagation of singularities \cite{MPT,RodinoWahlberg,SchulzWahlberg}, Fourier integral operators \cite{CSW,CGRV,Helffer}, and other arguments \cite{CdGN,deGo,EwSc}, just to name a very few. An extension of the Shubin calculus from Euclidean space to asymptotically conic manifolds has been developed in \cite{Krai}. 

Resolvents of \emph{differential} operators can be handled by means of  \emph{strongly} parameter-dependent symbols $a(z,\mu)$ satisfying 
 $$|\partial^\alpha_z \partial^j_\mu a(z,\mu)|\lesssim (1+|z|^2+\mu^2)^{(d-|\alpha|-j)/2},\qquad 
    (z,\mu)\in\rz^{2n}\times\rpbar,$$
and corresponding poly-homogeneous subclasses. Homogeneous components are in one-to-one correspondence with smooth functions on the unit-sphere in $\rz^{2n}\times\rpbar$. For example, $\mu^2+A$ with the harmonic oscillator $A$ is poly-homogeneous of order $d=2$ and $(\mu^2+A)^{-1}$ is strongly parameter-dependent of order $-2$.  

The use of strongly parameter-dependent symbols limits the resolvent construction to differential operators; however, it can be adapted to pseudodifferential operators, as has already been observed in the classical work \cite{Seel67} of Seeley on the complex powers of pseudodifferential operators on closed manifolds, see also Helffer's book  \cite{Helffer} for operators in the Shubin class. A systematic formulation in terms of a pseudodifferential calculus has been developed by Grubb in \cite{Grub} in the framework of H\"ormander classes, introducing the concept of \emph{regularity number}. The regularity number associated with a symbol measures how far or near it is from being strongly parameter-dependent. In \cite{Seil22} the author observed that the regularity number has an  interpretation in terms of a singularity structure of the involved homogeneous components. In the present paper we show that this interpretation extends to the Shubin class as well and allows us to construct a calculus of Shubin type with the following key properties: 
\begin{itemize}
 \item It contains both parameter-independent and strongly parameter-dependent Shubin classes; 
 \item ellipticity is characterized by the invertibility of two associated principal symbols (the homogeneous principal symbol 
  and the so-called \emph{principal limit symbol/operator}); 
 \item the parametrices to arbitrary elliptic elements (both differential and pseudodifferential) belong to the calculus and yield inverses 
 for large values of the parameter. 
\end{itemize}
The central idea is to work with classes of homogeneous symbols that are defined only on the punctured unit-sphere in $\rz^{2n}\times\rpbar$ with the ``north-pole'' $(z,\mu)=(0,1)$ removed; these symbols are singular in the north-pole in a controlled way, namely, when written in polar-coordinates centered in the north-pole, they admit a \emph{weighted Taylor expansion}. These classes yield a unified approach to both parameter-independent and parameter-dependent operators. 

As it turns out, homogeneous symbols in the new classes naturally  admit a second kind of symbol expansion which is closely related to the expansion used by Grubb and Seeley in their celebrated work \cite{GrSe} on resolvent trace expansions for boundary value problems with Atiyah-Patodi-Singer type boundary conditions. While the calculus of \cite{GrSe} has implemented an expansion of symbols in powers of the parameter $\mu$, in our approach arises an expansion in powers of $|(z,\mu)|$. This implies the existence of an expansion in the sense of Grubb-Seeley, cf. Theorem \ref{prop:first-expansion}, and yields another key feature of the calculus: 
\begin{itemize}
 \item Any element of order less than $-2n$ is $\mu$-wise of trace class in $L^2(\rz^n)$ and, for $\mu\to+\infty$, its trace admits 
  an expansion in decreasing powers of $\mu$, and decreasing powers of $\mu$ multiplied with $\log\mu$;   
\end{itemize}
the precise statement is given in Theorem \ref{thm:trace-expansion-01}. The resolvent trace expansion for elliptic pseudodifferential operators in the Shubin class is a special case, see Theorem \ref{thm:trace-expansion-02}. Recall that the resolvent trace expansion determines the pole structure of the $\zeta$-function and the short-time heat kernel asymptotics. Let us mention that the $\eta$-function of Shubin operators has been studied in \cite{Lope}. 

For simplicity of exposition, throughout the paper we shall focus on scalar-valued symbols; the extension to matrix-valued symbols does not cause essential new difficulties.  The matrix-valued calculus then allows to consider resolvents of Toeplitz-type operators, i.e., operators of the form $\bfA(\mu)=P_1A(\mu)P_0$ where $P_0$ and $P_1$ are zero-order idempotents in the Shubin class and $A(\mu)$ is from the new calculus. There is a natural notion of ellipticity for such Toeplitz operators and it holds:   
\begin{itemize}
 \item If $\bfA(\mu)$ is elliptic, there exists a parametrix of the form $\bfB(\mu)=P_0B(\mu)P_1$ such that 
  $\bfB(\mu)\bfA(\mu)=P_0$ and $\bfA(\mu)\bfB(\mu)=P_1$ for large values of $\mu$. 
\end{itemize}
This means that $\bfB(\mu)$ induces the inverse of the operator 
 $$P_1A(\mu): P_0(\calH^s_p)\lra P_1(\calH^s_p),\qquad s\in\rz,\;1<p<+\infty,$$
where $\calH^s_p$ represents here $L_p$-Sobolev spaces of Shubin-type, cf. \eqref{eq:spaces}, and $P_j(\calH^s_p)$ is the closed subspace  determined by the projection $P_j$. In order to keep the paper short, we will not discuss Toeplitz-operators in more detail, but refer the reader to \cite{Seil22}, where this was done in the context of operators from H\"ormander's class, see also \cite{Seil12,Seil15}. The adaption to the Shubin class is straight-forward. 

%%%%%%%%%%%%%%%%%%%%%%%%%%%%%%%%%%%%%%%%
%%%%%%%%%%%%%%%%%%%%%%%%%%%%%%%%%%%%%%%%
\section{Various classes of pseudodifferential symbols} 
\label{sec:02}

Let us recall a number of pseudodifferential symbol classes; essentially these classes 
have been introduced in \cite{Seil22,Seil24}, though we discuss here some additional 
aspects, in particular, in Sections \ref{sec:02.4} and \ref{sec:03}.  
Throughout the whole section $d$ and $\nu$ will denote real numbers. 

%%%%%%%%%%%%%%%%%%%%%%%%%%%%%%%%%%%%%%%%
\subsection{Hörmander classes} \label{sec:02.1}

We shall write $S^d_{1,0}(\rz^{m})$ for the vector space of all smooth  
$a(z):\rz^{m}\to\cz$ satisfying the uniform estimates
  $$|D^\alpha_z a(z)|\lesssim \spk{z}^{d-|\alpha|},\qquad \spk{z}:=(1+|z|^2)^{1/2},$$
for every multi-indices $\alpha\in\nz_0^{m}$. 
A symbol $a(z)\in S^d_{1,0}(\rz^{m})$ is called \emph{poly-homogeneous} 
if there exist smooth functions $a^{(d-\ell)}(z)$ defined for $z\not=0$ 
which are homogeneous of degree $d-\ell$ in $z$, i.e. 
  $$a^{(d-\ell)}(tz)=t^{d-\ell}\,a^{(d-\ell)}(z)\qquad\forall\;t>0,\;z\not=0,$$
such that 
 $$r_{a,N}(z):=a(z)-\chi(z)\sum_{\ell=0}^{N-1}a^{(d-\ell)}(z)\in 
   S^{d-N}_{1,0}(\rz^{m})$$
for every $N$, where $\chi(z)$ is an arbitrary fixed zero-excision function. 
$a^{(d)}(z)$ is the so-called homogeneous principal symbol of $a$. 
The space of poly-homogeneous symbols is denoted by $S^d(\rz^{m})$,  
the space of homogeneous symbols of degree $d$ by $S^d_\Hom(\rz^{m})$.

\begin{remark}\label{rem:topolgy}
$S^d_{1,0}(\rz^m)$ is a Fréchet space with topology induced by the semi-norms $\|a\|^{(\alpha)}_d=\sup_{z\in\rz^m}|\partial^\alpha a(z)|\spk{z}^{|\alpha|-d}$, $\alpha\in\nz_0^m$. Similarly for $S^d_\Hom(\rz^m)$, where in the supremum $z\not=0$ and $\spk{z}^{|\alpha|-d}$ is replaced by $|z|^{|\alpha|-d}$. The poly-homogeneous symbols of order $d$ then carry the projective topology induced by the maps $a\mapsto a^{(d-\ell)}: S^d(\rz^m)\to S^{d-\ell}_\Hom(\rz^m)$ and $a\mapsto r_{a,N}:S^d(\rz^m)\to S^{d-N}_{1,0}(\rz^m)$ with $\ell,N\in\nz_0$.  
\end{remark}

In an analogous way all symbol classes introduced below will be Frèchet spaces. 
We shall use this fact freely throughout the paper. 
Note that 
 $$S^{-\infty}(\rz^{m}):=\cap_{d\in\rz}S^d_{1,0}(\rz^{m})=\scrS(\rz^m)$$
coincides with the Schwartz space of rapidly decreasing functions on $\rz^m$. 

In the same way, replacing above $\rz^m$ by $\rz^{m}\times\rpbar$ and $z$ by $(z,\mu)$, we obtain 
the parameter-dependent Hörmander classes $S^d_{1,0}(\rz^{m}\times\rpbar)$, $S^d(\rz^{m}\times\rpbar)$
and $S^r_\Hom(\rz^{m}\times\rpbar)$, respectively.  Alternatively, 
 $$S^d_{1,0}(\rz^{m}\times\rpbar)=\Big\{a|_{\rz^{m}\times\rpbar}\mid a\in S^d_{1,0}(\rz^{m+1})\Big\}$$
and similar for the spaces of homogeneous and poly-homogeneous symbols. 
 
%%%%%%%%%%%%%%%%%%%%%%%%%%%%%%%%%%%%%%%%
\subsection{Symbols with regularity number} \label{sec:02.2}

The following definitions are motivated by \cite{Grub}. 
Let $\wt S^{d,\nu}_{1,0}(\rz^{m}\times\rpbar)$ denote the vector space of all smooth functions $a(z,\mu):\rz^{m}\times\rpbar\to\cz$ satisfying the uniform estimates
  $$|D^\alpha_z D^j_\mu a(z,\mu)|\lesssim \spk{z}^{\nu-|\alpha|}\spk{z,\mu}^{d-\nu-j}$$
for every multi-indices $\alpha\in\nz_0^{m}$ and every $j\in\nz_0$.  
%Note that 
% $$\wt S^{d-1,\nu-1}_{1,0}(\rz^{m}\times\rpbar)\cup\wt S^{d,\nu+1}_{1,0}(\rz^{m}\times\rpbar)
%     \subset \wt S^{d,\nu}_{1,0}(\rz^{m}\times\rpbar).$$
A function $a(z,\mu)$ belongs, by definition, to $\wt S^{d,\nu}_\Hom(\rz^{m}\times\rpbar)$ if it is defined 
and smooth for $z\not=0$,  
  $$a(tz,t\mu)=t^{d}\,a(z,\mu)\qquad\forall\;t>0,\;\mu\ge0,\;z\not=0,$$
and satisfies the uniform estimates  
  $$|D^\alpha_z D^j_\mu a(z,\mu)|\lesssim |z|^{\nu-|\alpha|}|(z,\mu)|^{d-\nu-j}$$
for derivatives of any order. 
A symbol $a(z,\mu)\in \wt S^{d,\nu}_{1,0}(\rz^{m})$ is called \emph{poly-homogeneous} 
if there exist $a^{(d-\ell,\nu-\ell)}(z,\mu)\in \wt S^{d-\ell,\nu-\ell}_\Hom(\rz^{m}\times\rpbar)$ such that 
 $$r_{a,N}(z):=a(z,\mu)-\chi(z)\sum_{\ell=0}^{N-1}a^{(d-\ell,\nu-\ell)}(z,\mu)\in 
    \wt S^{d-N,\nu-N}_{1,0}(\rz^{m}\times\rpbar)$$
for every $N$, where $\chi(z)$ is an arbitrary fixed zero-excision function (note that for the parameter-dependent H\"ormander classes an excision function $\chi(z,\mu)$ is used). The leading component $a^{(d,\nu)}$ is the so-called homogeneous principal symbol of $a$. The class of such poly-homogeneous symbols is denoted by $\wt S^{d,\nu}(\rz^{m}\times\rpbar)$. The class of smoothing symbols is 
 $$\wt S^{d-\infty,\nu-\infty}(\rz^{m}\times\rpbar):=\cap_{N\ge 0}\; \wt S^{d-N,\nu-N}_{1,0}(\rz^{m}\times\rpbar);$$
note that $r$ belongs to this class if and only if 
  $$|D^\alpha_z D^j_\mu r(z,\mu)|\lesssim \spk{z}^{-L}\spk{\mu}^{d-\nu-j}$$
for all $L\ge0$ and derivatives of any order, i.e., $r$ behaves in $z$ as a rapidly decreasing function while in $\mu$ as a Hörmander symbol of order $d-\nu$. For this reason we shall also write 
\begin{equation}\label{eq:remainder01}
 \wt S^{d-\infty,\nu-\infty}(\rz^{m}\times\rpbar)=S^{d-\nu}_{1,0}(\rpbar,S^{-\infty}(\rz^{m})).
\end{equation}
Let us also observe that 
 $$\wt S^{d-1,\nu-1}_{1,0}(\rz^{m}\times\rpbar)\cup
     \wt S_{1,0}^{d,\nu+1}(\rz^{m}\times\rpbar)
     \subset \wt S_{1,0}^{d,\nu}(\rz^{m}\times\rpbar)$$
and analogously for the subclasses of poly-homogeneous symbols. 
%%%%%%%%%%%%%%%%%%%%%%%%%%%%%%%%%%%%%%%%
\subsection{Symbols with expansion at infinity} \label{sec:02.3}

Let $\wtbfS^{d,\nu}_{1,0}(\rz^{m}\times\rpbar)$ denote the subspace of 
$\wt S^{d,\nu}_{1,0}(\rz^{m}\times\rpbar)$ consisting of those symbols $a$ for which esists a 
sequence of symbols $a^\infty_{[d,\nu+j]}\in S^{\nu+j}_{1,0}(\rz^{m})$ such that 
\begin{equation}\label{eq:expansion01}
 a(z,\mu)-\sum_{j=0}^{N-1} a^\infty_{[d,\nu+j]}(z)\,[z,\mu]^{d-\nu-j}\in 
 \wt S^{d,\nu+N}_{1,0}(\rz^{m}\times\rpbar)
\end{equation}
for every choice of $N$. In \eqref{eq:expansion01}, $[z,\mu]$ denotes a smooth positive function in $(z,\mu)$ which coincides with 
$|(z,\mu)|$ whenever $|(z,\mu)|\ge1$. 
In particular, $[z,\mu]^\rho\in S^\rho(\rz^m\times\rpbar)$ for arbitrary $\rho$. 
The symbols $a^\infty_{[d,\nu+j]}$ are uniquely determined by $a$, in particular, 
\begin{equation*}
 a^\infty_{[d,\nu]}(z)=\lim_{\mu\to+\infty}[z,\mu]^{\nu-d}a(z,\mu)
\end{equation*}
with convergence in $S^{\nu+1}_{1,0}(\rz^m)$. We shall call $a^\infty_{[d,\nu]}$ the \emph{principal limit symbol} of $a$. 

As a simple example, observe that if $a(z)\in S^d_{1,0}(\rz^m)$ is independent of the parameter $\mu$, then $a\in \wtbfS^{d,d}_{1,0}(\rz^{m}\times\rpbar)$ and $a^\infty_{[d,d]}=a$, i.e., the expansion in \eqref{eq:expansion01} has only one term.

Note that if $a\in\wtbfS^{d,\nu}_{1,0}(\rz^{m}\times\rpbar)$ then 
$\partial^k_\mu\partial^\alpha_z a\in \wtbfS^{d-k-|\alpha|,\nu-|\alpha|}_{1,0}(\rz^{m}\times\rpbar)$ 
with 
 $$(\partial^k_\mu\partial^\alpha_z a)^\infty_{[d-k-|\alpha|,\nu-|\alpha|]}=
     (d-\nu)(d-\nu-1)\cdots(d-\nu-k+1)D^\alpha_z a^\infty_{[d,\nu]}.$$
Moreover, if $a_\ell\in\wtbfS^{d_\ell,\nu_\ell}_{1,0}(\rz^{m}\times\rpbar)$ for $\ell=0,1$ then 
$a_1a_0\in \wtbfS^{d_0+d_1,\nu_0+\nu_1}_{1,0}(\rz^{m}\times\rpbar)$ with 
 $$(a_1a_0)^\infty_{[d_0+d_1,\nu_0+\nu_1]}=a^\infty_{1,[d_1,\nu_1]}a^\infty_{0,[d_0,\nu_0]}.$$  

\begin{lemma}\label{lem:lower-order}
\begin{itemize}
 \item[a$)$] Let $\ell\in\nz$ and $a\in \wtbfS^{d,\nu}_{1,0}(\rz^{m}\times\rpbar)$. Then 
   $$a\in \wtbfS^{d,\nu+\ell}_{1,0}(\rz^{m}\times\rpbar)\iff a_{[d,\nu]}^\infty=a_{[d,\nu+1]}^\infty=\ldots=a_{[d,\nu+\ell-1]}^\infty=0.$$ 
 \item[b$)$] Let $a\in \wtbfS^{d-\ell,\nu-\ell}_{1,0}(\rz^{m}\times\rpbar)$ for some $\ell\in\nz$. 
   Then $a\in \wtbfS^{d,\nu}_{1,0}(\rz^{m}\times\rpbar)$ and  
     $$a^\infty_{[d,\nu+j]}=a^\infty_{[d-\ell,\nu-\ell+j]}\in S^{\nu-\ell+j}(\rz^m)\qquad\forall\;j\in\nz_0.$$
\end{itemize}
\end{lemma}
\begin{proof}
a) is obvious by definition of the symbol classes. b) is valid since  
 $$a(z,\mu)-\sum_{j=0}^{N-1} a^\infty_{[d-\ell,\nu-\ell+j]}(z)[z,\mu]^{d-\nu-j}\in \wt S^{d-\ell,\nu-\ell+N}_{1,0}(\rz^{m}\times\rpbar)$$
%    \subset \wt S^{d,\nu+N}_{1,0}(\rz^{m}\times\rpbar)$$
and $\wt S^{d-\ell,\nu-\ell+N}_{1,0}(\rz^{m}\times\rpbar)    \subset \wt S^{d,\nu+N}_{1,0}(\rz^{m}\times\rpbar)$ for every $N$. 
\end{proof}

The following result on asymptotic summation is \cite[Theorem 5.5]{Seil22} (in the special case of $x$-independent symbols). 

\begin{theorem}\label{thm:asymptotic-summation}
Let $a_\ell\in\wtbfS^{d-\ell,\nu-\ell}_{1,0}(\rz^m\times\rpbar)$, $\ell\in\nz_0$. 
Then there exists a symbol $a\in\wtbfS^{d,\nu}_{1,0}(\rz^m\times\rpbar)$ such that, for every $N$, 
 $$a-\sum_{\ell=0}^{N-1}a_\ell\in \wtbfS^{d-N,\nu-N}_{1,0}(\rz^m\times\rpbar).$$ 
Moreover, 
$a_{[d,\nu+j]}^\infty\sim\sum\limits_{\ell=0}^{+\infty}a_{\ell,[d-\ell,\nu-\ell+j]}^\infty$
asymptotically in $S^{\nu+j}_{1,0}(\rz^m)$ for every $j$. The symbol $a$ is unique 
modulo $\wtbfS^{d-\infty,\nu-\infty}(\rz^m\times\rpbar)$. 
\end{theorem}

It can be shown, see \cite[Proposition 3.6]{Seil24}, that 
\begin{equation}\label{eq:smoothing}
 \wtbfS^{d-\infty,\nu-\infty}(\rz^m\times\rpbar)=
 \cap_{N\ge 0}\; \wtbfS^{d-N,\nu-M}_{1,0}(\rz^m\times\rpbar)=
 S^{d-\nu}(\rpbar,S^{-\infty}(\rz^m)),
\end{equation}
i.e., a smoothing symbol behaves as a rapidly decreasing function in $z$ and as a \textit{poly-homogeneous} symbol of order $d-\nu$ in $\mu$, cf. \eqref{eq:remainder01}.

%%%%%%%%%%%%%%%%%%%%%%%%%%%%%%%%%%%%%%%%
\subsection{Poly-homogeneous symbols with expansion at infinity} \label{sec:02.4}

Let  
\begin{equation*}
 \wh\sz^m_+=\sz^m_+\setminus\{(0,1)\}
    =\{(z,\mu)\in\rz^m\times\rpbar\mid |(z,\mu)|=1,\;z\not=0\},
\end{equation*}
denote the unit-sphere in $\rz^m\times\rpbar$ with the ``north-pole'' $(0,1)$ removed. 

\begin{definition}
Let $C^{\infty,\nu}_\bfT(\wh\sz^m_+)$, $\nu\in\rz$, denote the vector space of all $\wh a\in C^{\infty}(\wh\sz^m_+)$ such that 
 $$(r,\phi)\mapsto r^{-\nu}\,\wh a\big(r\phi,\sqrt{1-r^2}\big):(0,1)_r\times\sz^{m-1}_\phi\lra\cz$$
extends to a smooth function on $[0,1)\times\sz^{m-1}$ $($smoothness including $r=0)$; 
 $$\wh a_{\spk\nu}(\phi):=\lim_{r\to0+}r^{-\nu}\,\wh a\big(r\phi,\sqrt{1-r^2}\big),\qquad\phi\in\sz^{m-1},$$
is called the \emph{angular symbol} of $\wh a$. 
\end{definition}

In the previous definition, $\sz^{m-1}$ is the unit-sphere in $\rz^{m}_z$. In other words, the function $|z|^{-\nu}\wh a(z,\mu)$ 
written in polar-coordinates around the north-pole $(0,1)$ admits a Taylor expansion centered in $r=0$. 

Note that if $\wh a_p(z,\mu)= |z|^{\nu+j}p(z/|z|)$ with $p\in C^\infty(\sz^{m-1})$, then 
$\wh a_p\in C^{\infty,\nu}_\bfT(\wh\sz^m_+)$, since $\wh a_p\big(r\phi,\sqrt{1-r^2}\big)=r^{\nu+j}p(\phi)$. This yields the following: 

\begin{lemma}\label{lem:Taylor}
Let $\wh a\in C^{\infty,\nu}_\bfT(\wh\sz^m_+)$. Then there exist uniquely determined symbols $p_j\in C^\infty(\sz^{m-1})$ such that
 $$\wh a(z,\mu)=\sum_{j=0}^{N-1} |z|^{\nu+j}p_j\Big(\frac{z}{|z|}\Big)+\wh r_N(z,\mu),\qquad \wh r_N\in C^{\infty,\nu+N}_\bfT(\wh\sz^m_+),$$
for every $N\in\nz$. Clearly, $p_0=\wh a_{\spk\nu}$. 
\end{lemma}
 
Homogeneous extension of degree $d$ leads to the classes $\wtbfS^{d,\nu}_\Hom(\rz^m\times\rpbar)$, consisting of all functions $a$ of the form 
 $$a(z,\mu)=|(z,\mu)|^d \;\wh a\Big(\frac{(z,\mu)}{|(z,\mu)|}\Big),\qquad 
    \wh a\in C^{\infty,\nu}_\bfT(\wh\sz^m_+).$$
Note that any such homogeneous function is, in general, only defined for $z\not=0$. 
The \emph{angular symbol} of $a$ is, by definition,  
\begin{equation}\label{eq:principal-angular-symbol}
   a_{\spk\nu}(z)=|z|^\nu\; \wh a_{\spk\nu}\Big(\frac{z}{|z|}\Big)\in S^\nu_\Hom(\rz^m).
\end{equation}
It can be shown that $\wtbfS^{d,\nu}_\Hom(\rz^m\times\rpbar)\subset\wt S^{d,\nu}_\Hom(\rz^m\times\rpbar)$, cf. \cite[Theorem 4.4]{Seil22}. If $\wh a_p(z,\mu)= |z|^{\nu+j}p(z/|z|)$ as above, its $d$-homogeneous extension is 
 $$a(z,\mu)=|z|^{\nu+j}p(z/|z|)|(z,\mu)|^{d-\nu-j}.$$ 
Together with  Lemma \ref{lem:Taylor} we thus obtain: 

\begin{proposition}\label{prop:homogeneous-limit-expansion}
Let $a\in\wtbfS^{d,\nu}_\Hom(\rz^m\times\rpbar)$. Then there exist uniquely determined symbols 
$a^\infty_{[d,\nu+j]}\in S^{\nu+j}_\Hom(\rz^m)$ such that 
 $$a(z,\mu)=\sum_{j=0}^{N-1} a^\infty_{[d,\nu+j]}(z)|(z,\mu)|^{d-\nu-j}
     +r_N(z,\mu),\qquad r_N\in \wt S^{d,\nu+N}_\Hom(\rz^m\times\rpbar),$$
for every $N\in\nz$. In particular, $\wh a_{\spk\nu}=a^\infty_{[d,\nu]}$. 
\end{proposition}

Since the spaces of homogeneous symbols from H\"ormander's class $S^d_\Hom(\rz^m\times\rpbar)$ 
are in one-to-one correspondence with the functions that are smooth on the entire unit-sphere $\sz^m_+\subset\rz^m\times\rpbar$, 
it is obvious that  
 $$S^d_\Hom(\rz^m\times\rpbar)\subset \wtbfS^{d,0}_\Hom(\rz^m\times\rpbar)$$
and that the angular symbol of $a\in S^d_\Hom(\rz^m\times\rpbar)$ 
is just its value in $(0,1)$, i.e., $a_{\spk 0}(z)\equiv a(0,1)$. 

\begin{corollary}\label{cor:classical}
Let $a\in \wtbfS^{d,\nu}_\Hom(\rz^m\times\rpbar)$ and $\chi(z)$ be a zero-excision function. Then $\chi a\in \wtbfS^{d,\nu}_{1,0}(\rz^m\times\rpbar)$ with limit symbol 
$(\chi a)^\infty_{[d,\nu]}=\chi a_{\spk\nu}$. 
\end{corollary}
\begin{proof}
Let $a$ have an expansion as in Proposition \ref{prop:homogeneous-limit-expansion}. 
Choose a zero-excision function $\kappa(z,\mu)$ such that 
$\kappa(z,\mu)\chi(z)=\chi(z)$ on $\rz^m\times\rpbar$. Then 
\begin{align*}
 \chi(z)|(z,\mu)|^{d-\nu-j}=\chi(z)[z,\mu]^{d-\nu-j}
    +\chi(z)\big(\kappa(z,\mu)|(z,\mu)|^{d-\nu-j}-[z,\mu]^{d-\nu-j}\big).
\end{align*}
The second term on the right-hand side is compactly supported in $(z,\mu)$. This immediately yields the claim. 
\end{proof}

\begin{definition}\label{def:poly-homogeneous}
$\wtbfS^{d,\nu}(\rz^m\times\rpbar)$ is the subspace of $\wtbfS^{d,\nu}_{1,0}(\rz^m\times\rpbar)$ consisting of all symbols $a$ which admit a sequence of symbols  $a^{(d-j,\nu-j)}\in \wtbfS^{d-j,\nu-j}_\Hom(\rz^m\times\rpbar)$ such that, for every $N$, 
 $$a-\chi\sum_{j=0}^{N-1}a^{(d-j,\nu-j)}\in \wtbfS^{d-N,\nu-N}_{1,0}(\rz^m\times\rpbar);$$
here $\chi(z)$ is an arbitrary zero-excision function. We call $a^{(d,\nu)}$ the homogeneous principal symbol of $a$, while $a_{\spk{\nu}}:=(a^{(d,\nu)})_{\spk{\nu}}$ is the \emph{principal angular symbol} of $a$. 
\end{definition}

\begin{proposition}[\text{\cite[Proposition 3.13]{Seil24}}]
If $a\in S^{d}(\rz^{m}\times\rpbar)$ then $a\in\wtbfS^{d,0}(\rz^{m}\times\rpbar)$ with $a_{[d,0]}(z)=a_{\spk 0}(z)=a^{(d)}(0,1)$, i.e., 
the value of the homogeneous principal symbol of $a$ in $(0,1)$. 
\end{proposition}

If $a\in \wtbfS^{d,\nu}(\rz^m\times\rpbar)$ is as in the previous definition then, by Corollary \ref{cor:classical} and Lemma \ref{lem:lower-order}.b), 
\begin{align}\label{eq:asymptotic-expansion}
 a_{[d,\nu]} \sim \sum_{j=0}^{+\infty}(\chi a^{(d-j,\nu-j)})_{[d-j,\nu-j]}
 \sim\chi \sum_{j=0}^{+\infty} (a^{(d-j,\nu-j)})_{\spk{\nu-j}}
\end{align}
asymptotically in $S^\nu_{1,0}(\rz^m)$. In particular, the principal limit symbol $a_{[d,\nu]}$ is poly-homogeneous and its 
homogeneous principal coincides with the principal angular symbol of $a$. 

\begin{lemma}\label{prop:lower-order}
Let $a\in\wtbfS^{d,\nu}(\rz^m\times\rpbar)$ such that $a^{(d,\nu)}=0$ and $a^\infty_{[d,\nu]}=0$, i.e., 
both homogeneous principal symbol and principal limit symbol vanish. Then $a\in\wtbfS^{d-1,\nu}(\rz^m\times\rpbar)$
\end{lemma}
\begin{proof}
Let $a$ be as in Definition \ref{def:poly-homogeneous} and let $a_j:=a^{(d-j,\nu-j)}$ be the $(d-j)$-homogeneous extension of $\wh a_j\in C^{\infty,\nu-j}_\bfT(\wh\sz^m)$. Let $b:=a$ and $b_j:=a_{j+1}$ for $j\in\nz_0$. 

In view of \eqref{eq:asymptotic-expansion} and Lemma \ref{lem:lower-order}.b), $a^\infty_{[d,\nu]}=0$ implies both $\wh{a}_{j,\spk{\nu-j}}=0$ and $(\chi a_j)^\infty_{[d,\nu]}=0$ for every $j$. The first yields that $\wh a_j\in C^{\infty,\nu-j+1}(\wh\sz^{m}_+)$. Therefore,   $$b_j(z,\mu)=|(z,\mu)|^{d-1-j}\wh a_{j+1}\Big(\frac{(z,\mu)}{|(z,\mu)|}\Big)\in \wtbfS^{d-1-j,\nu-j}_\Hom (\rz^m\times\rpbar).$$
Moreover, since $a_0=0$,  
 $$b-\chi\sum_{j=0}^{N-1}b_j=a-\chi\sum_{j=0}^{N}a_j=:r_N\in \wtbfS^{d-N-1,\nu-N-1}(\rz^m\times\rpbar).$$
The second observation from above then implies that $r_{N,[d-N-1,\nu-N-1]}^\infty=r_{N,[d,\nu]}^\infty=0$, hence $r_N\in\wtbfS^{d-1-N,\nu-N}(\rz^m\times\rpbar)$ 
due to Lemma \ref{lem:lower-order}.a). Thus $a=b$ is as claimed.  
\end{proof}

\begin{remark}\label{rem:blow-up}
For another characterization of $C^{\infty,\nu}_\bfT(\wh\sz^m_+)$ let us consider  
  $$\mathbb{B}=\{(w,\mu)\in\rz^m\times\rpbar\mid (|w|-1)^2+\mu^2=1\},$$
obtained by ``blowing up'' $\wh\sz^m_+$ at the north-pole. 
$\mathbb{B}$ is a smooth manifold with boundary, with the two boundary components 
$\partial\mathbb{B}_j=\{(w,\mu)\in \mathbb{B}\mid \mu=j\}$, $j=0,1$. 
\begin{center}
\begin{tikzpicture}
\draw[dashed] (0,0) arc (0:180:1.2cm and 0.25cm);
\draw (-2.4,0) arc (180:360:1.2cm and 0.2cm);
\draw (0,0) arc (0:180:1.2cm);
\draw[dashed] (6,0) arc (0:180:1.8cm and 0.25cm);
\draw (2.4,0) arc (180:360:1.8cm and 0.2cm);
\draw (6,0) arc (0:74:1.25cm);
\draw (2.4,0) arc (180:106:1.25cm);
\draw[color=red!60, thick] (3.3,1.2) arc (180:360:0.9cm and 0.1cm);
\draw[color=red!60, thick] (5.1,1.2) arc (0:180:0.9cm and 0.15cm);
\draw[dashed] (-1.2,-0.2) -- (-1.2,1.2);
\draw[->] (-1.2,1.2) -- (-1.2,1.6);
\draw (-1.2,-0.2) -- (-1.2,-0.6);
\draw[dashed] (-2.4,0) -- (0,0);
\draw (-3.0,0) -- (-2.4,0);
\draw[->] (0,0) -- (0.6,0);
\draw[dashed] (4.2,-0.2) -- (4.2,1.0);
\draw[->] (4.2,1) -- (4.2,1.6);
\draw (4.2,-0.2) -- (4.2,-0.6);
\draw[dashed] (2.4,0) -- (6,0);
\draw (1.8,0) -- (2.4,0);
\draw[->] (6,0) -- (6.6,0);
\node at (0.6,-0.2) {$z$};
\node at (-0.97,1.6) {$\mu$};
\node at (6.6,-0.2) {$w$};
\node at (4.43,1.6) {$\mu$};
\filldraw [color=red!60] (-1.2,1.2) circle (1.7pt);
\end{tikzpicture}
\captionof{figure}{The punctured sphere $\wh\sz^m_+$ and the manifold $\mathbb{B}$}
\end{center}
The map $\chi(z,\mu)=\big(z+\frac{z}{|z|},\mu\big)$ induces a diffeomorphism between 
$\wh\sz^m_+$ and $\mathbb{B}\setminus \partial\mathbb{B}_1$. Then $\wh a$ belongs to $C^{\infty,\nu}_\bfT(\wh\sz^m_+)$ 
if and only if there exists $\wh a_b\in C^\infty(\mathbb{B})$ such that 
 $$|z|^{-\nu}\,\wh a(z,\mu)=(\wh a_b\circ\chi)(z,\mu),\qquad (z,\mu)\in\wh\sz^m_+.$$ 
In this case,  $\wh a_{\spk\nu}(\phi)=\wh a_b(\phi,1)$ for every $\phi\in\sz^{m-1}$. 
\end{remark}

The previous remark allows to easily obtain the following result: 

\begin{lemma}\label{lem:invertibility-homogenous}
A symbol $a\in \wtbfS^{d,\nu}_\Hom(\rz^m\times\rpbar)$ is pointwise invertible with inverse belonging to $\wtbfS^{-d,-\nu}_\Hom(\rz^m\times\rpbar)$ if and only if $a$ is pointwise invertible on $\wh\sz^m_+$ and the principal angular symbol is pointwise invertible on $\sz^{m-1}$. 
\end{lemma}

In fact, if $a$ is the $d$-homogeneous extension of $\wh a$ and $\wh a$ is related to $\wh a_b$ as described in Remark \ref{rem:blow-up}, the two properties in the lemma are just equivalent to the fact that $\wh a_b\in C^\infty(\mathbb{B})$ is pointwise everywhere invertible on $\mathbb{B}$ (including the boundary), hence $\wh a_b^{-1}$ exists and belongs to  $C^\infty(\mathbb{B})$.

%%%%%%%%%%%%%%%%%%%%%%%%%%%%%%%%%%%%%%%%
\section{Relation with a calculus of Grubb and Seeley}\label{sec:03}

In this section we establish a relationship with the symbol classes introduced  in \cite{GrSe} and derive an asymptotic expansion formula which is the core for the resolvent trace asymptotics in Section \ref{sec:05}. In the following we shall use the notation $m_+:=\max(m,0)$. 
Let us also introduce the polynomials  
\begin{equation}\label{eq:coefficient-polynomials}
 p_{\rho,\ell}(z)=\frac{1}{\ell!}\partial_t^\ell \spk{tz}^\rho\big|_{t=0}
\end{equation}
where $\rho\in\rz$ and $\ell\in\nz_0$; they are homogeneous of degree $\ell$ and vanish whenever $\ell$ is odd. Obviously $p_{\rho,0}=1$. 

\begin{lemma}\label{lem:exp}
Let $\rho\in\rz$ and $L\in\nz$. Then, for all $z\in\rz^m$ and all $\mu>0$, 
 $$|(z,\mu)|^\rho=\sum_{\ell=0}^{L-1}p_{\rho,\ell}(z)\mu^{\rho-\ell}+\gamma_{\rho,L}(z,\mu),$$
with $p_{\rho,\ell}$ as in $\eqref{eq:coefficient-polynomials}$ and  
 $$\gamma_{\rho,L}(z,\mu)= \sum_{|\alpha|=L} z^\alpha\int_0^1(1-\theta)^{L-1} q_{\rho,\alpha}(\theta z,\mu)\,d\theta$$
with certain symbols $q_{\rho,\alpha}(z,\mu)\in S^{\rho-|\alpha|}_\Hom(\rz^m\times\rpbar)$.  
In particular, $\gamma_{\rho,L}$ is homogeneous of degree $\rho$ on $\rz^m\times\rz_+$ and 
$|\gamma_{\rho,L}(z,\mu)|\lesssim \mu^{\rho-L}|z|^L$ on $\rz^m\times\rz_+$ whenever $L\ge\rho$. 
\end{lemma}
\begin{proof}
First note that  $|(z,t^{-1})|^{\rho}=t^{-\rho}|(tz,1)|^{\rho}=t^{-\rho}\spk{tz}^\rho$ for every $t>0$. 
Thus,  with certain constants $c_{\rho,\alpha}$, 
 $$\partial^k_t\spk{tz}^\rho=\sum_{|\alpha|=k}c_{\rho,\alpha}z^\alpha  |(tz,1)|^{\rho}_\alpha,\qquad  
     |(z,\mu)|^{\rho}_\alpha:=\partial^\alpha_z|(z,\mu)|^\rho.$$
Hence, by Taylor expansion and by the homogeneity of degree $\rho-|\alpha|$ of $|(z,\mu)|^{\rho}_\alpha$, 
 $$\spk{tz}^{\rho}=\sum_{\ell=0}^{L-1}
     p_{\rho,\ell}(z)t^{\ell}+\frac{1}{(L-1)!}t^{L}  r_{\rho,L}(t,z),$$
with 
\begin{align}\label{eq:remainder-term}
\begin{split}
 r_{\rho,L}(t,z)
     &=\sum_{|\alpha|=L}c_{\rho,\alpha}z^\alpha
    \int_0^1(1-\theta)^{L-1} |\theta t z,1|^\rho_\alpha\,d\theta\\
     &=\sum_{|\alpha|=L}c_{\rho,\alpha}z^\alpha
    \int_0^1(1-\theta)^{L-1} |\theta z,t^{-1}|^\rho_\alpha\,d\theta\,t^{\rho-L}. 
\end{split}
\end{align}
The claim follows by substituting $t$ with $\mu^{-1}$.
\end{proof}

The following result is an analogue of  \cite[Theorem 1.12]{GrSe}. 

 \begin{theorem}\label{prop:first-expansion}
Let $a\in \wtbfS^{d,\nu}_{1,0}(\rz^m\times\rpbar)$. 
% and set $p(t,z)=\big([z,\mu]^{\nu-d}a(z,\mu)\big)\big|_{\mu=1/t}$ for $t\in(0,1]$. 
Then there exist $a^\infty_{\{d,\nu+k\}}\in S^{\nu+k}_{1,0}(\rz^m)$ such that, for every $k\in\nz_0$, 
\begin{equation}\label{eq:second-limit-symbol}
 a^\infty_{\{d,\nu+k\}}(z)=\frac{1}{k!}\lim_{t\to 0^+}\partial^k_t \big(t^{d-\nu}a(z,t^{-1})\big)
\end{equation}
with convergence in $S^{\nu+r_k}_{1,0}(\rz^m)$ with $r_k=\max(d-\nu,k+1)$. Moerover, 
\begin{equation}\label{eq:limit02}
  \mu^{N-d+\nu}\Big[a(z,\mu)-\sum_{k=0}^{N-1} a^\infty_{\{d,\nu+k\}}(z)\mu^{d-\nu-k}\Big]\in S^{\nu+r_N}_{1,0}(\rz^m)
\end{equation}
uniformly in $\mu\ge1$ for every $N\in\nz$. In particular, $a^\infty_{\{d,\nu\}}=a^\infty_{[d,\nu]}$.
\end{theorem}
\begin{proof}
Let $0<t\le 1$. 
For every $N\ge(d-\nu)_+$, there exists an $r_N\in \wt S^{d,\nu+N}_{1,0}(\rz^m\times\rpbar)$ such that 
 $$t^{d-\nu}a(z,t^{-1})=\sum_{j=0}^{N-1}a_{[d,\nu+j]}(z)\spk{tz}^{d-\nu-j}t^j+ t^{d-\nu}r_N(z,t^{-1}).$$
Since $\partial_t^k\big(t^{d-\nu}r_N(z,t^{-1})\big)$ is a linear combination of terms 
$t^{d-\nu-k}((\mu\partial_\mu)^\ell r_n)(z,t^{-1})$ with $0\le\ell\le k$, 
\begin{align*}
  |\partial_t^k D^\alpha_z (t^{d-\nu}r_N(z,1/t)\big)|
  &\lesssim t^{d-\nu-k}\spk{z}^{\nu+N-|\alpha|}[z,t^{-1}]^{d-\nu-N}\\
  &=t^{N-k}\spk{z}^{\nu+N-|\alpha|}\spk{tz}^{d-\nu-N}\lesssim t^{N-k}\spk{z}^{\nu+N-|\alpha|}.
\end{align*}
This shows $\partial_t^k(t^{\nu-d}r_N(\cdot,t^{-1}))\xrightarrow{t\to0^+}0$ in $S^{\nu+N}_{1,0}(\rz^m)$ for every $0\le k\le N-1$. 
As verified in the proof of Lemma \ref{lem:exp} (with $\rho=d-\nu-j$ and $L=N-j)$, 
 $$t^j\spk{tz}^{d-\nu-j}=\sum_{\ell=0}^{N-j-1}p_{d-\nu-j,\ell}(z)t^{\ell+j}+ t^Nr_{jN}(t,z),,$$
where $r_{jN}(t,z)$ is linear combination of terms 
 $$z^\alpha\int_0^1(1-\theta)^{N-j-1}q(t\theta z)\,d\theta,\quad 
     |\alpha|=N-j,\;q\in S^{d-\nu-N}_{1,0}(\rz^m)\subset S^0_{1,0}(\rz^m)$$
(cf. the first identity in \eqref{eq:remainder-term}). Since $\partial_t^k (t^Nq(t\theta z))= 
t^{N-k}\wt q(t\theta z)$ with a resulting $\wt q\in S^0_{1,0}(\rz^m)$, it follows that 
$\partial^k_t(t^Nr_{jN}(t,\cdot))\xrightarrow{t\to0^+}0$ in $S^{N-j}_{1,0}(\rz^m)$ whenever $0\le k\le N-1$.  
Altogether,
\begin{align*}
 t^{d-\nu} a(z,t^{-1})
%&=\sum_{j=0}^{N-1}\sum_{\ell=0}^{N-j-1}p_{d-\nu-j,\ell}(z)a_{[d,\nu+j]}(z)t^{j+\ell}+ R_N(z,t)\\
          &=\sum_{k=0}^{N-1}\Big(\sum_{\ell=0}^k p_{d-\nu-\ell,k-\ell}(z)a_{[d,\nu+\ell]}^\infty(z)\Big)t^{k}+ R_N(z,t),
\end{align*}
where $\partial_t^k R_N(t,\cdot)\xrightarrow{t\to0^+}0$ in $S^{\nu+N}_{1,0}(\rz^m)$ for every $0\le k\le N-1$. 
\forget{
It follows that 
\begin{equation}\label{eq:coefficients}
 p_k(z)=\sum_{\ell=0}^{k}\alpha_{k\ell}z^{k-\ell}a_{[d,\nu+\ell]}(z),\qquad \alpha_{k\ell}=\beta_{\ell,k-\ell} k!,
\end{equation}
with convergence in \eqref{eq:limit} is in $S^{\nu+N}(\rz^m)$ for all $0\le k\le N-1$. 
}
Since $N\ge(d-\nu)_+$ is arbitrary, the first claim follows. 
\eqref{eq:limit02} follows by substituting $t$ with $\mu^{-1}$. 
\end{proof}

As a consequence of the previous proof, the coefficient symbols $a^\infty_{\{d,\nu+k\}}$ can be determined 
from the $a_{[d,\nu+k]}^\infty$ and vice versa: Let  
$B=B_{d-\nu,N}(z)=(b_{k\ell}(z))_{0\le k,\ell\le N-1}$, be the lower triangular matrix with 
 $$b_{k\ell}(z)=
    \begin{cases}
     \frac{1}{(k-\ell)!}\partial^{k-\ell}_t \spk{tz}^{d-\nu-\ell}\big|_{t=0}&:0\le \ell\le k, \\ 0& :\text{otherwise}
    \end{cases},\qquad 0\le k\le N-1.$$
%$\beta(d-\nu-\ell,k-\ell)|z|^{k-\ell}$ if $0\le\ell\le k$ and $b_{jk}=0$ otherwise. 
If $\vec{a}^\infty_{\{d,\nu\}}=(a^\infty_{\{d,\nu\}},\ldots,a^\infty_{\{d,\nu+N-1\}})$ and 
$\vec{a}^\infty_{[d,\nu]}=(a^\infty_{[d,\nu]},\ldots,a^\infty_{[d,\nu+N-1]})$ then 
$\vec{a}^\infty_{\{d,\nu\}}=B\vec{a}^\infty_{[d,\nu]}$. Since the diagonal entries 
of $B$ are equal to $1$, $B$ is invertible and $B^{-1}$ has the same structure.

The next theorem is the analogue of  \cite[Theorem 2.1]{GrSe}. 

\begin{theorem}\label{thm:trace-expansion-01}
Let $a\in\wtbfS^{d,\nu}(\rz^m\times\rpbar)$ with $d<-m$.  
Then, for resulting coefficients depending on $a$,  
 $$\int_{\rz^m} a(z,\mu)\,dz\sim_{\mu\to+\infty} \sum_{j=0}^{+\infty}c_j\mu^{d+m-j}+\sum_{\ell=0}^{+\infty}
    \big(c_\ell^\prime\log\mu+c_\ell^{\prime\prime}\big)\mu^{d-\nu-\ell}.$$
\end{theorem}

As will result from the proof, the coefficients $c_\ell'$ are non-zero at most if $\ell=j-m-\nu$ for some $j\in\nz_0$; these coefficients as well as the coefficients $c_j$ only depend on the values of the homogeneous components of $a^{(d-j,\nu-j)}(z,\mu)$, $|z|\ge1$.  The proof itself is essentially the same as that given in \cite{GrSe}. 

\begin{proof}
Let $N\ge(d-\nu)_+$ be arbitrary and fixed. Fix $J\in\nz$ with $\nu+N-J<-m$ and write  
 $$a(z,\mu)=\chi(z)\sum_{j=0}^{J-1} a^{(d-j,\nu-j)}(z,\mu) + r(z,\mu),
    \qquad r\in \wtbfS^{d-J,\nu-J}_{1,0}(\rz^m\times\rpbar),$$
where $\chi$ is a zero-excision function which is constant $1$ outside of the unit-ball. 
Due to Proposition \ref{prop:first-expansion}, 
\begin{equation}\label{eq:r-expansion} 
 r(z,\mu)=\sum_{\ell=0}^{N-1} r^\infty_{\{d-J,\nu-J+\ell\}}(z)\mu^{d-\nu-\ell}
    +\mu^{d-\nu-N}r'(z,\mu)
\end{equation}
with $r^\infty_{\{d-J,\nu-J+\ell\}}\in S^{\nu-J+\ell}_{1,0}(\rz^m)$ and $r'(\cdot,\mu)\in S^{\nu-J+N}_{1,0}(\rz^m)$ uniformly in $\mu\ge 1$. Hence, by the choice of $J$, 
 $$\int_{\rz^m} r(z,\mu)\,dz=\sum_{\ell=0}^{N-1} c^{\prime\prime}_{\ell}\mu^{d-\nu-\ell}
     +O(\mu^{d-\nu-N}),\qquad \mu\ge1.$$
Due to the choice of $\chi$, for $\mu\ge1$, 
 $$\int_{|z|\ge\mu} \chi(z)a^{(d-j,\nu-j)}(z,\mu)\,dz=\mu^{d-j+m}\int_{|z|\ge1} a^{(d-j,\nu-j)}(z,1)\,dz=c_j^{(0)} \mu^{d+m-j}.$$
Next fix $L$ with $L>J-1-m-\nu$ $($in particular, $L\ge N$ and $\nu-j+L>-m$ for every $j=0,\ldots,J-1$; the latter implies that 
$|z|^{\nu-j+L}$ is locally integrable on $\rz^m)$. 
Due to Proposition \ref{prop:homogeneous-limit-expansion} we can write  
 $$a^{(d-j,\nu-j)}(z,\mu)=\sum_{k=0}^{L-1} a^{\infty}_{[d-j,\nu-j+k]}(z)|(z,\mu)|^{d-\nu-k}+r_{j}(z,\mu)$$ 
with $a^{\infty}_{[d-j,\nu-j+k]}(z)\in S^{\nu-j+k}_\Hom(\rz^m)$ and $r_{j}\in \wtbfS^{d-j,\nu-j+L}_\Hom(\rz^m\times\rpbar)$. 
Observe that, for $\mu>0$ and $z\not=0$,
\begin{align*}  
 |r_{j}(z,\mu)|&\lesssim |z|^{\nu-j+L}|(z,\mu)|^{d-\nu-L}
     =\mu^{d-\nu-L}|z|^{\nu-j+L}\spk{z\mu^{-1}}^{d-\nu-L}\\
     &\le \mu^{d-\nu-L}|z|^{\nu-j+L}.
\end{align*}
Therefore, by using Lemma \ref{lem:exp} (with $\rho=d-\nu-k$ and $L$ replaced by $L-k)$,
\begin{equation}\label{eq:for-later}
 a^{(d-j,\nu-j)}(z,\mu)=\sum_{\ell=0}^{L-1} q_{j,\ell}(z)\mu^{d-\nu-\ell}+s_{j}(z,\mu)
\end{equation}
with $q_{j,\ell}(x,\xi)\in S^{\nu-j+\ell}_\Hom(\rz^m)$ and a remainder $s_{j}$ which is homogeneous of degree 
$d-j$ on $(\rz^m\setminus\{0\})\times\rz_+$ and satisfies $|s_{j}(z,\mu)|\lesssim \mu^{d-\nu-L}|z|^{\nu-j+L}$.  
Hence 
\begin{align*}
 \int_{0\le |z|\le\mu}\chi(z)s_{j}(z,\mu)\,dz
  &=\int_{0\le |z|\le\mu}s_{j}(z,\mu)\,dz-\int_{0\le |z|\le1}(1-\chi)(z)s_{j}(z,\mu)\,dz\\
  &=c_j^{(1)}\mu^{d+m-j}+O(\mu^{d-\nu-L}).
\end{align*}
Finally, 
 $$\int_{0\le |z|\le\mu}\chi(z)q_{j,\ell}(z)\,dz=\int_{1\le |z|\le\mu}q_{j,\ell}(z)\,dz+c_{j,\ell}''$$
and, by homogeneity and by using polar-coordinates, 
 $$\mu^{d-\nu-\ell}\int_{1\le |z|\le\mu}q_{j,\ell}(z)\,dz=
     \begin{cases}
      \alpha_{j,\ell}(\mu^{d+m-j}-\mu^{d-\nu-\ell}) &: \ell\not=j-m-\nu\\
      \alpha_{j,\ell}\mu^{d-\nu-\ell}\log\mu=\alpha_{j,\ell}\mu^{d+m-j}\log\mu &: \ell=j-m-\nu
     \end{cases}
 $$
This yields the desired expansion and completes the proof. 
\end{proof}

%%%%%%%%%%%%%%%%%%%%%%%%%%%%%%%%%%%%%%%%
%%%%%%%%%%%%%%%%%%%%%%%%%%%%%%%%%%%%%%%%
\section{New classes of Shubin type pseudodifferential operators} \label{sec:04}

Substituting above $m=2n$ and the variable $z\in\rz^m$ by $(x,\xi)\in\rz^{2n}$ leads to the definition of various symbol classes of Shubin type: 

\begin{definition}
The spaces $\Gamma^d_{1,0}(\rz^{2n})$ and $\Gamma^d(\rz^{2n})$ consist of all symbols  
$a(x,\xi):\rz^n\times\rz^n\to\cz$ belonging to $S^d_{1,0}(\rz^{2n}_{z=(x,\xi)})$ and $S^d(\rz^{2n}_{z=(x,\xi)})$, 
respectively. 

The spaces $\Gamma^{d}_{1,0}(\rz^{2n}\times\rpbar)$, $\wt\Gamma^{d,\nu}_{1,0}(\rz^{2n}\times\rpbar)$, and $\wtbfG^{d,\nu}_{1,0}(\rz^{2n}\times\rpbar)$ 
consist of all symbols $a(x,\xi,\mu):\rz^n\times\rz^n\times\rpbar\to\cz$ belonging to 
$S^{d}_{1,0}(\rz^{2n}_{z=(x,\xi)}\times\rpbar)$, $\wt S^{d,\nu}_{1,0}(\rz^{2n}_{z=(x,\xi)}\times\rpbar)$, and $\wtbfS^{d,\nu}_{1,0}(\rz^{2n}_{z=(x,\xi)}\times\rpbar)$, respectively. 
Analogously we define the subspaces of poly-homogeneous symbols $\Gamma^{d}(\rz^{2n}\times\rpbar)$, $\wt\Gamma^{d,\nu}(\rz^{2n}\times\rpbar)$, and $\wtbfG^{d,\nu}(\rz^{2n}\times\rpbar)$.
\end{definition}

The classes of homogeneous symbols are defined analogously. 

Evidently, the previous definition is tautological because, as function spaces, the $\Gamma$-spaces simply coincide with the corresponding $S$-spaces. However, we prefer to introduce the notation of the $\Gamma$-classes in order to emphasize that for these classes we employ the standard quantization map to pass from a symbol to its associated $\psi$do on $\rz^n$, where we have a distinction in the role of the $x$- and $\xi$-variable, namely   
 $$[a(x,D)u](x)=\int e^{ix\xi}a(x,\xi)\wh{u}(\xi)\,\dbar\xi,\qquad x\in\rz,\;\scrS(\rz^n).$$
Here, $\wh u$ is the Fourier transform of $u$, $\dbar\xi=(2\pi)^{-n}\,d\xi$, and the symbol may also depend on $\mu$; in this case write $a(x,D,\mu)$ for the associated operator-family. Note that if $a$ is a parameter-dependent symbol of order $d$ in any of the above introduced symbol classes and $a_\mu(x.\xi):=a(x,\xi.\mu)$ then $a_\mu\in \Gamma^d_{1,0}(\rz^{2n})$ for every $\mu$. 

\begin{remark}
With each $a\in \wtbfG^{d,\nu}(\rz^{2n}\times\rpbar)$ there are associated three principal symbols, the homogeneous principal symbol, the principal limit symbol, and the principal angular symbol $($always defined as in Section $\ref{sec:02}$ by substituting $z$ with $(x,\xi))$. With the limit symbol $a_{[d,\nu]}\in \Gamma^d(\rz^{2n})$ we associate the so-called principal limit operator $a_{[d,\nu]}(x,D)$. 
\end{remark}

Initially defined on rapidly decreasing functions, the $\psi$do associated with symbols of order $d$ in one of the above $\Gamma$-classes extend by continuity to linear and bounded maps 
 $$\calH^s_p(\rz^n)\lra \calH^{s-d}_p(\rz^n),\qquad s\in\rz,\;1<p<+\infty,$$
where the scale of Shubin-type Sobolev spaces is defined by 
\begin{equation}\label{eq:spaces}
 \calH^s_p(\rz^n)=\{u\in\scrS'(\rz^n)\mid \spk{x,D}^s u\in L^p(\rz^n)\}
\end{equation}
with the canonical norm. Note that $s$ measures simultaneously the smoothness and the decay/growth at infinity of a distribution. 
Note that $\cap_s \calH^s_p(\rz^n)=\scrS(\rz^n)$, the Schwartz space of rapidly decreasing functions. 

If $a\in \Gamma^{d}_{1,0}(\rz^{2n})$, $d\ge 0$, is elliptic 
% i.e.,there exists a $b\in \Gamma^{-d}_{1,0}(\rz^{2n})$ such that $a\#b-1$ and $b\#a-1$ are symbols of arbitrarily negative order $($any such $b$ is a so-called parametrix of $a)$. Then, 
then, by elliptic regularity, the unbounded operator 
 $$A=a(x,D):\scrS(\rz^n)\subset\calH^s_p(\rz^n)\lra \calH^s_p(\rz^n)$$
has a unique closed extension given by the action of $a(x,D)$ on the domain $\calH^{s+d}_p(\rz^n)$.  

For later purpose let us recall that the embedding $\calH^s_2(\rz^n)\hookrightarrow \calH^r_2(\rz^n)$ is of trace class 
provided $s-r>2n$. In particular, the following result is known to be true:

\begin{proposition}
If $a\in \Gamma^d_{1,0}(\rz^{2n})$ with $d<-2n$, then $a(x,D)$ induces a trace class operator 
in $L^2(\rz^n)=\calH^0_2(\rz^n)$. Moreover, $\displaystyle\mathrm{tr}\,a(x,D)=\iint a(x,\xi)\, dx\dbar\xi$. 
\end{proposition}

%%%%%%%%%%%%%%%%%%%%%%%%%%%%%%%%%%%%%%%%
\subsection{Composition and formal adjoint} \label{sec:03.1}

On symbolic level, the composition of two $\psi$do $a_0(x,D)$ and $a_1(x,D)$ corresponds to the so-called 
\emph{Leibniz product} of symbols,  
\begin{align}\label{eq:leibniz-product}
 (a_1\#a_0)(x,\xi)=\iint e^{iy\eta}a_1(x,\xi+\eta)a_0(x+y,\xi)\,dy\dbar\eta 
\end{align}
$($integration in the sense of oscillatory integrals, cf. for example \cite{Kuma}$)$. It is known that 
the Leibniz product induces maps 
 $$\Gamma^{d_1}_{1,0}(\rz^{2n})\times\Gamma^{d_0}_{1,0}(\rz^{2n})\lra \Gamma^{d_0+d_1}_{1,0}(\rz^{2n}),\qquad d_0,d_1\in\rz;$$
poly-homogeneity is preserved under composition. 
%By the usual asymptotic expansion formula, i.e., 
% $$a_1\#a_0\equiv \sum_{|\alpha|=0}^{N-1}\frac{1}{\alpha!}(\partial^\alpha_\xi a_1)(D^\alpha_x a_0)\quad\mod 
%    \Gamma^{d_0+d_1-N}_{1,0}(\rz^{2n})$$
%for every $N\in\nz$, poly-homogeneity is preserved under composition. 
The analogous statement holds for the parameter-dependent classes $\Gamma^{d}_{1,0}(\rz^{2n}\times\rpbar)$, where the composition, 
respectively the Leibniz product, is taken pointwise for each $\mu$. 

\begin{theorem}\label{thm:composition}
The $(\mu$-wise$)$ Leibniz-product induces maps
% $$\wt\Gamma^{d_0,\nu_0}_{1,0}(\rz^{2n}\times\rpbar)\times\wt\Gamma_{1,0}^{d_1,\nu_1}(\rz^{2n}\times\rpbar)\lra 
%     \wt\Gamma_{1,0}^{d_0+d_1,\nu_0+\nu_1}(\rz^{2n}\times\rpbar)$$
%as well as 
 $$\wtbfG_{1,0}^{d_1,\nu_1}(\rz^{2n}\times\rpbar)\times \wtbfG_{1,0}^{d_0,\nu_0}(\rz^{2n}\times\rpbar)\lra 
     \wtbfG_{1,0}^{d_0+d_1,\nu_0+\nu_1}(\rz^{2n}\times\rpbar).$$
For arbitrary $N\in\nz_0$, 
\begin{equation}\label{eq:composition-expansion}
 a_1\#a_0\equiv \sum_{|\alpha|=0}^{N-1}\frac{1}{\alpha!}(\partial^\alpha_\xi a_1)(D^\alpha_x a_0)\quad\mod 
    \wtbfG_{1,0}^{d_0+d_1-2N,\nu_0+\nu_1-2N}(\rz^{2n}\times\rpbar)
\end{equation}
%with arbitrary $N$. This yields the claims for poly-homogeneous symbols. 
%The analogous statement is true for the subclasses of poly-homogeneous symbols. 
Moreover, 
 $$(a_1\#a_0)_{[d_0+d_1,\nu_0+\nu_1]}^\infty=a_{1,[d_1,\nu_1]}^\infty\#a_{0,[d_0,\nu_0]}^\infty.$$
If both $a_0$ and $a_1$ are poly-homogeneous then %in the sense of Definition $\ref{aaa}$ 
so is $a_1\#a_0$ and 
 $$(a_1\#a_0)^{(d_0+d_1,\nu_0+\nu_1)}=a_1^{(d_1,\nu_1)} a_0^{(d_0,\nu_0)}.$$
In other words, homogeneous principal symbol, principal limit symbol, and principal angular symbol are multiplicative under composition.  
\end{theorem}
\begin{proof}
The proof relies on results of \cite[Section 8]{Seil22}. Let first $a\in\wtbfS^{d,\nu}_{1,0}(\rz^{2n}\times\rpbar)$ be an arbitrary 
symbol and define 
 $$q(\wt z;z,\mu)=a(z+\wt z,\mu),\qquad z,\wt z\in\rz^{2n}.$$
By \cite[Proposition 8.2]{Seil22} (in the special case of a symbol with constant coefficients), 
$\wt z\mapsto q(\wt z)$ is a so-called amplitude function on $\rz^{2n}_{\wt z}$ with values in 
$\wtbfS^{d,\nu}_{1,0}(\rz^{2n}_z\times\rpbar)$ and $\wt z\mapsto q(\wt z)_{[d,\nu]}=a_{[d,\nu]}(z+\wt z)$ is an amplitude function 
with values in $S^{\nu}_{1,0}(\rz^{2n}_z)$. Writing $\wt z=(y,\eta)$ and $z=(x,\xi)$, it follows that 
 $$p(y,\eta;x,\xi,\eta):=a_1(x,\xi+\eta,\mu)a_0(x+y,\xi,\mu)=a_1(z+(0,\eta),\mu)a_0(z+(y,0),\mu)$$
defines an amplitude function of $(y,\eta)\in\rz^{2n}$ with values in $ \wtbfG_{1,0}^{d_0+d_1,\nu_0+\nu_1}(\rz^{2n}\times\rpbar)$ and with 
 $$\big(p(y,\eta)_{[d_0+d_1,\nu_0+\nu_1]}\big)(x,\xi)=a_{1.[d_1,\nu_1]}(x,\xi+\eta)a_{0,[d_0,\nu_0]}(x+y,\xi)$$
being an amplitude function of $(y,\eta)\in\rz^{2n}$ with values in $\Gamma^{\nu_0+\nu_1}_{1,0}(\rz^{2n})$.  It follows that 
\begin{equation}\label{eq:composition-integral}
 (a_1\#a_0)(x,\xi,\mu)=\iint e^{iy\eta}a_1(x,\xi+\eta,\mu)a_0(x+y,\xi,\mu)\,dy\dbar\eta
\end{equation}
converges as an oscillatory integral with values in $ \wtbfG_{1,0}^{d_0+d_1,\nu_0+\nu_1}(\rz^{2n}\times\rpbar)$ and that its limit symbol 
is the oscillatory integral 
 $$\iint e^{iy\eta}a_{1,[d_1,\nu_1]}(x,\xi+\eta)a_{0,[d_0,\nu_0]}(x+y,\xi)\,dy\dbar\eta=(a_{1,[d_1,\nu_1]}^\infty\#a_{0,[d_0,\nu_0]}^\infty)(x,\xi),$$
which converges in $\Gamma^{\nu_0+\nu_1}_{1,0}(\rz^{2n})$. In a similar way, see \cite[Theorem 8.4]{Seil22}, one verifies \eqref{eq:composition-expansion}. This yields the additional claims for poly-homogeneous symbols. 
\end{proof}

\begin{remark}\label{rem:limit-symbols}
Let the assumptions be as in Theorem $\ref{thm:composition}$. Using \eqref{eq:composition-integral} and 
\eqref{eq:second-limit-symbol}, it is obvious that, for arbitrary $k\in\nz_0$, 
  $$(a_1\#a_0)^\infty_{\{d_0+d_1,\nu_0+\nu_1+k\}}
      =\sum_{\ell=0}^k a^\infty_{1,\{d_1,\nu_1+\ell\}}\#a^\infty_{0,\{d_0,\nu_0+k-\ell\}}.$$
\end{remark}

The formally adjoint operator $A^{(*)}$ of $A=a(x,D)$ is determined by the relation 
 $$(Au,v)_{L^2(\rz^{n})}=(u,A^{(*)}v)_{L^2(\rz^n)},\qquad u,v\in\scrS(\rz^n);$$
$A^{(*)}$ is a pseudodifferential operator with symbol 
 $$a^{(*)}(x,\xi)=\iint e^{-iy\eta}\overline{a(x+y,\xi+\eta)}\,dy\dbar\eta.$$
It is well-known that $a\mapsto a^{(*)}:\Gamma^d_{1,0}(\rz^{2n})\to\Gamma^d_{1,0}(\rz^{2n})$ and that poly-homogeneity is preserved.  Again, for $\mu$-dependent symbols one defines the formal adjoint pointwise for each $\mu$. Then a corresponding statement holds for $\Gamma^d_{1,0}(\rz^{2n}\times\rpbar)$ and, similarly to Theorem \ref{thm:composition}, one can prove the following: 

\begin{theorem}\label{thm:adjoint}
Taking $(\mu$-wise$)$ the formal adjoint induces maps
% $$\wt\Gamma^{d_0,\nu_0}_{1,0}(\rz^{2n}\times\rpbar)\times\wt\Gamma_{1,0}^{d_1,\nu_1}(\rz^{2n}\times\rpbar)\lra 
%     \wt\Gamma_{1,0}^{d_0+d_1,\nu_0+\nu_1}(\rz^{2n}\times\rpbar)$$
%as well as 
 $$\wtbfG_{1,0}^{d,\nu}(\rz^{2n}\times\rpbar)\lra  \wtbfG_{1,0}^{d,\nu}(\rz^{2n}\times\rpbar).$$
%The analogous statement is true for the subclasses of poly-homogeneous symbols. 
Moreover, 
$(a^{(*)})_{[d,\nu]}^\infty=(a_{[d,\nu]}^\infty)^{(*)}$. 
If $a$ is poly-homogeneous then so is $a^{(*)}$ with  
$(a^{(*)})^{(d,\nu)}=\overline{a^{(d,\nu)}}$ and $(a^{(*)})_{\spk{\nu}}=\overline{a_{\spk{\nu}}}$. 
\end{theorem}

Concerning the symbols of \eqref{eq:second-limit-symbol}, it is clear that 
$(a^{(*)})^\infty_{\{d,\nu+k\}}=(a^\infty_{\{d,\nu+k\}})^{(*)}$ for every $k$.

%%%%%%%%%%%%%%%%%%%%%%%%%%%%%%%%%%%%%%%%
\subsection{Ellipticity, parametrices, and inverses} \label{sec:03.3}

A symbol $a\in \wtbfG^{d,\nu}(\rz^{2n}\times\rpbar)$ is called elliptic, if 
\begin{itemize}
 \item[(E1)] the homogeneous principal symbol $a^{(d,\nu)}$ is pointwise invertible on $\wh\sz^{2n}_+$, 
 \item[(E2)] the limit operator $a_{[d,\nu]}^\infty(x,D)$ is invertible. 
\end{itemize}
In (E2), invertibility refers to invertibility as a map $\calH^s(\rz^{n})\to\calH^{s-\nu}(\rz^n)$ for some $s\in\rz$. Due to spectral invariance of the Shubin class, the invertibility is independent of the choice of $s$ and is equivalent to the fact that the inverse of the limit-operator is a $\psi$do with symbol in $\Gamma^{-\nu}(\rz^{2n})$. A further equivalent formulation of (E2) is to say that the principal limit symbol has an inverse with respect to the Leibniz product. 

\begin{theorem}\label{thm:parametrix01}
Let $a\in \wtbfG^{d,\nu}(\rz^{2n}\times\rpbar)$. Then $a$ is elliptic if and only if there exists a parametrix $b\in\wtbfG^{-d,-\nu}(\rz^{2n}\times\rpbar)$, i.e.,  $a\#b-1,b\#a-1\in \scrS(\rpbar;\Gamma^{-\infty}(\rz^{2n}))$. In this case $a(x,D,\mu)$ is invertible for sufficiently large values of $\mu$ and the parametrix can be constructed in such a way that,   
for some non-negative constant $R=R(a)$, 
 $$b(x,D,\mu)=a(x,D,\mu)^{-1}\qquad\forall\;\mu\ge R.$$
%for some non-negative constant $R=R(a)$. 
\end{theorem}

\begin{proof}
By (E2), the limit operator is, in particular, elliptic, i..e., has invertible homogeneous principal symbol. The latter coincides with the principal angular symbol of $a$. Thus, as observed in Lemma  \ref{lem:invertibility-homogenous}, $a^{(d,\nu)}$ is invertible with inverse belonging to 
$\wtbfG^{-d,-\nu}_\Hom(\rz^{2n}\times\rpbar)$. Therefore there exists a $b_1\in \wtbfG^{-d,-\nu}(\rz^{2n}\times\rpbar)$ 
such that 
\begin{equation}\label{eq:parametrix01}
 1-a\#b_1=:r_1\in \wtbfG^{0-\infty,0-\infty}(\rz^{2n}\times\rpbar).
\end{equation}
In fact, $\chi(x,\xi)(a^{(d,\nu)})^{-1}$ yields a $b_1$ with resulting $r_1\in \wtbfG^{-1,-1}(\rz^{2n}\times\rpbar)$; the standard von Neumann series argument then allows to modify $b_1$ in such a way that \eqref{eq:parametrix01} ist true. The principal limit symbol $r_{1,[0,0]}^\infty$ is smoothing and thus is also $(a_{[d,\nu]}^\infty)^{-\#}\#r_{1,[0,0]}^\infty$; here the superscript $-\#$ means the inverse with respect to the Leibniz product. Hence 
 $$s_1(x,\xi,\mu):=((a_{[d,\nu]}^\infty)^{-\#}\#r_{1,[0,0]}^\infty)(x,\xi)\;[x,\xi,\mu]^{-d+\nu}
    \in \wtbfG^{-d-\infty,-\nu-\infty}(\rz^{2n}\times\rpbar)$$ 
and $s_1$ has limit symbol $s_{1,[-d,-\nu]}^\infty=(a_{[d,\nu]}^\infty)^{-\#}\#r_{1,[0,0]}^\infty$. 
Setting $b_2:=b_1+s_1$ it follows that $b_2$ satisfies the analogue of \eqref{eq:parametrix01} with $r_2\in \wtbfG^{0-\infty,0-\infty}(\rz^{2n}\times\rpbar)$ and 
 $$r_{2,[0,0]}^\infty=(1-a\#b_1)_{[0,0]}^\infty-a_{[d,\nu]}^\infty\#s_{1,[-d,-\nu]}^\infty=r_{1,[0,0]}^\infty
     -r_{1,[0,0]}^\infty=0.$$
Thus $r_2\in \wtbfG^{-1-\infty,0-\infty}(\rz^{2n}\times\rpbar)$ by Proposition \ref{prop:lower-order}.  Since 
 $$\wtbfG^{-1-\infty,0-\infty}(\rz^{2n}\times\rpbar)=S^{-1}(\rpbar;\Gamma^{-\infty}(\rz^{2n})),$$
cf. \eqref{eq:smoothing}, another application of the Neumann series argument allows to modify $b_2$ in such a way that 
 $$r_2\in \cap_{N\ge 0}\;\wtbfG^{-N-\infty,0-\infty}(\rz^{2n}\times\rpbar)=
     \scrS(\rpbar;\Gamma^{-\infty}(\rz^{2n})).$$
This yields that $1-r_2(x,D,\mu)$ is invertible for large $\mu$ with 
 $$(1-r_2(x,D,\mu))^{-1}=1+r_2(x,D,\mu)+r_2(x,D,\mu)(1-r_2(x,D,\mu))^{-1}r_2(x,D,\mu).$$
If $\chi(\mu)$ is an 0-excision function vanishing on a sufficiently large neighborhood of $0$, 
 $$s_2:=-r_2-\chi r_2\#(1-r_2)^{-\#}\#r_2\in \scrS(\rpbar;\Gamma^{-\infty}(\rz^{2n}))$$
and $(1-r_2(x,D,\mu))^{-1}=1-s_2(x,D,\mu)$ for large $\mu$. It follows that $b:=b_2\#(1-s_2)$ 
is a parametrix which is a right-inverse for large $\mu$. In the same way we can construct a 
parametrix which is a left-inverse. It follows that $b$ is as desired.  
\end{proof}

%%%%%%%%%%%%%%%%%%%%%%%%%%%%%%%%%%%%%%%%
%%%%%%%%%%%%%%%%%%%%%%%%%%%%%%%%%%%%%%%%
\section{Resolvents and their trace expansion}\label{sec:05}

In order to discuss resolvents it is useful to introduce a further symbol class. 
Define, for $d\in\gz$ and $\nu\in\nz_0$,  
\begin{equation}
 \bfGa^{d,\nu}(\rz^{2n}\times\rpbar)=
 \Gamma^d(\rz^{n}\times\rpbar)+\wtbfG^{d,\nu}(\rz^{n}\times\rpbar),
\end{equation}
a non-direct sum of Fréchet spaces. Note that 
$\bfGa^{d,\nu}(\rz^n\times\rpbar)\subset\wtbfG^{d,0}(\rz^{n}\times\rpbar)$ with equality in case $\nu=0$. Furthermore, 
$ \Gamma^{d}(\rz^{2n}\times\rpbar)=\cap_{\nu\ge0}\, \bfGa^{d,\nu}(\rz^{2n}\times\rpbar)$. 

With $a=a_0+\wt a\in \bfGa^{d,\nu}(\rz^n\times\rpbar)$ we associate the homogneous principal symbol 
\begin{equation}
 a^{(d,0)}=a_0^{(d)}+\wt a^{(d,\nu)}
\end{equation}
%defined on the punctured semi-sphere $\wh\sz^{2n}_+$, 
and the principal limit symbol 
\begin{equation}
 a^\infty_{[d,0]}=a_0^{(d)}(0,0,1)+ \wt a^\infty_{[d,0]};
\end{equation}
note that the first term on the right-hand side is the value of the homogeneous principal symbol of $a_0$ in $(x,\xi,\mu)=(0,0,1)$ while the second term equals zero whenever $\nu\ge1$. Both principal symbols are the ones from the class $\wtbfG^{d,0}(\rz^{n}\times\rpbar)$. 

\begin{remark}\label{rem:positive-regularity}
In case $a=a_0+\wt a\in\bfGa^{d,\nu}(\rz^{2n}\times\rpbar)$ with $\nu\ge1$, the homogeneous principal symbol $a^{(d,0)}$, initially defined on the punctured semi-sphere $\wh\sz^{2n}_+$, extends by continuity to a continuous function on $\sz_+^{2n}$ by setting  $a^{(d,0)}(0,0,1)=a_0^{(d)}(0,0,1)$. 
\end{remark}

Using that $\Gamma^d(\rz^{n}\times\rpbar)\subset \wtbfG^{d,0}(\rz^{n}\times\rpbar)$ and Theorems \ref{thm:composition} and \ref{thm:adjoint}, it is straightforward to verify the following: 

\begin{theorem}
Let $d,d_0.d_1\in\gz$ and $\nu_0,\nu_1\in\nz_0$. 
Let $a\in \bfGa^{d,\nu}(\rz^{2n}\times\rpbar)$ and $a_\ell\in\bfGa^{d_\ell,\nu_\ell}(\rz^{2n}\times\rpbar)$ for $\ell=0,1$. Then 
 $$a_1\#a_0\in \bfGa^{d_0+d_1,\nu}(\rz^{2n}\times\rpbar),\qquad\nu=\min(\nu_0,\nu_1),$$
and both homogeneous principal symbol and principal limit symbol/operator behave multiplicative under composition. 
Moreover, $a^{(*)}\in \bfGa^{d,\nu}(\rz^{2n}\times\rpbar)$ with principal symbols behaving as in Theorem $\ref{thm:adjoint}$. 
\end{theorem}

A symbol $a\in \bfGa^{d,\nu}(\rz^{2n}\times\rpbar)$ is called elliptic, if 
\begin{itemize}
 \item[(E$1'$)] the homogeneous principal symbol $a^{(d,0)}$ is pointwise invertible on $\wh\sz^{2n}_+$, 
 \item[(E$2'$)] the principal limit operator $a^\infty_{[d,0]}(x,D)$ is invertible. 
\end{itemize}
In case $\nu=0$ these conditions coincide with (E1) and (E2) from above. In case $\nu\ge1$ the conditions can be merged into a single condition, namely 
\begin{itemize}
 \item[(E)] the homogeneous principal symbol $a^{(d,0)}$ is pointwise invertible on $\sz^{2n}_+$, 
\end{itemize}
where we have employed the continuous extension of the homogeneous principal symbol as described in Remark \ref{rem:positive-regularity}. 

\begin{lemma}\label{lem:inverse02}
Let $\nu\ge 1$ and $a=a_0+\wt a\in \bfGa^{d,\nu}(\rz^n\times\rpbar)$ be elliptic. Then 
 $$(a^{(d,0)})^{-1}\in \Gamma^d_\Hom(\rz^{n}\times\rpbar)+\wtbfG^{d,\nu}_\Hom(\rz^{n}\times\rpbar).$$
\end{lemma}
\begin{proof}
By multiplication with $[x,\xi,\mu]^{-d}$ we may assume $d=0$. Suppose first that $a^{(0)}_0\equiv1$. 
Then the function on the blown-up semi-sphere $\mathbb{B}$, cf. Remark \ref{rem:blow-up},  associated 
with $\sigma_\psi^{(0)}(a)$ is of the form $1+r$ with a smooth function $r$ that vanishes to order $\nu$ on the boundary component $\partial\mathbb{B}_1$.  Hence the inverse is of the same form. 

In the general case, note that the invertibility implies that $a^{(0)}_0(0,0,1)$ is invertible, hence $a^{(0)}_0$ restricted to $\sz^{2n}_+$ is invertible in a small neighborhood of the north-pole $(0,0,1)$. Thus we can construct a function $\wh b\in C^\infty(\sz^{2n}_+)$ which is pointwise everywhere invertible and coincides with $(a^{(0)}_0)^{-1}$ in a neighborhood of $(0,0,1)$. If $b$ is the homogeneous extension of order $0$ of $\wh b$, we find that $b\sigma_\psi^{(0)}(a)=1+\wt a_1$ is pointwise everywhere invertible and $\wt a_1\in \wtbfG_\Hom^{0,\nu}(\rz^n\times\rpbar)$. Hence $(a^{(d,0)})^{-1}=(1+\wt a_1)^{-1}b$ has the desired form, in view of the first part of the proof. 
\end{proof}

\begin{theorem}\label{thm:parametrix02}
Let $a\in \bfGa^{d,\nu}(\rz^{2n}\times\rpbar)$ with $d\in\gz$ and $\nu\in\nz_0$. Then $a$ is elliptic if and only if there exists a parametrix $b\in\bfGa^{-d,\nu}(\rz^{2n}\times\rpbar)$ $($note the same regularity number $\nu)$, i.e.,  $a\#b-1,b\#a-1\in \scrS(\rpbar;\Gamma^{-\infty}(\rz^{2n}))$. In this case $a(x,D,\mu)$ is invertible for sufficiently large values of $\mu$ and the parametrix can be constructed in such a way that  
$b(x,D,\mu)=a(x,D,\mu)^{-1}$ for large $\mu$. 
\end{theorem}
\begin{proof}
The case $\nu=0$ is Theorem \ref{thm:parametrix01}; hence we may assume $\nu\ge1$. 

Using Lemma \ref{lem:inverse02}, there exists a $b_0\in \bfGa^{-d,\nu}(\rz^{2n}\times\rpbar)$ such that $r_0:=1-a\#b_0\in \bfGa^{-1,\nu-1}(\rz^{2n}\times\rpbar)$. Since $\nu-1\ge0$, by asymptotic summation we can construct $r_1\in \bfGa^{-1,\nu-1}(\rz^{2n}\times\rpbar)$ such that $r_1-\sum_{1\le j\le N-1}r_0^{\#j}\in \bfGa^{-N,\nu-1}(\rz^{2n}\times\rpbar)$ for every $N$. Then $b_1:=b_0+b_0\#r_1$ belongs to  
$\bfGa^{-d,\nu}(\rz^{2n}\times\rpbar)
+\bfGa^{-d-1,\nu-1}(\rz^{2n}\times\rpbar)\subset \bfGa^{-d,\nu}(\rz^{2n}\times\rpbar)$ and 
$a\#b_1=1-r$ with $r\in\scrS(\rpbar,\Gamma^{-\infty}(\rz^{2n}))$. 
Now proceed as in the proof of Theorem \ref{thm:parametrix01}.  
\end{proof}

%%%%%%%%%%%%%%%%%%%%%%%%%%%%%%%%%%%%%%%%
\subsection{Resolvent trace expansions} \label{sec:04.1}

Let $p_0\in\Gamma^{d}(\rz^{2n})$, $d\in\nz$, and set $p(x,\xi,\mu)=\mu^d-p_0(x,\xi)$. 
Then $p\in\bfGa^{d,d}(\rz^{2n}\times\rpbar)$ and 
 $$p(x,D,\mu)=\mu^d-p_0(x,D).$$ 
Since $p^{(d,0)}(x,\xi,\mu)=\mu^d-p_0^{(d)}(x,\xi)$, clearly $p$ is elliptic (in the sense of (E) above) 
if and only if 
\begin{itemize}
 \item[$(\mathrm{E}')$] $p_0^{(d)}$ never assumes values in $\rpbar$.  
\end{itemize} 

Then Theorem \ref{thm:parametrix02} immediately yields the following result: 

\begin{theorem}\label{thm:resolvent}
Let $p_0\in\Gamma^{d}(\rz^{2n})$, $d\in\nz$, satisfy condition $(\mathrm{E}')$. Then there exists a parametrix 
$b\in\bfGa^{-d,d}(\rz^{2n}\times\rpbar)$ of $p$ and a $\lambda_0\ge0$ such that, for all 
$\lambda\ge\lambda_0$, 
 $$(\lambda-p_0(x,D))^{-1}=b(x,D,\lambda^{1/d}).$$
\end{theorem}

Since $\lambda b(x,\xi,\lambda^{1/d})\in \Gamma^{-d}_{1,0}(\rz^n)$ uniformly in $\lambda\ge0$, 
 $$\|\lambda(\lambda-p_0(x,D))^{-1}\|_{\scrL(\calH^s_p(\rz^n))}\lesssim1,\qquad\lambda\ge\lambda_0,$$ 
for every $s\in\rz$ and $1<p<+\infty$. Since with $p_0\in\Gamma^{d}(\rz^{2n})$ satisfying $(\mathrm{E}')$  
there exists an $\eps>0$ such that also $e^{i\theta}p_0$ satisfies $(\mathrm{E}')$ for all real $\theta$ with $|\theta|<\eps$, 
the resolvent of $p_0(x,D)$ exists for large values of $\lambda$ in a small sector containing $\rz_+$ and the 
resolvent estimate extends to this sector. In particular $p_0(x,D)$ generates an analytic semigroup in $\calH^s_p(\rz^n)$. 

\begin{theorem}\label{thm:trace-expansion-02}
Let $p_0\in\Gamma^{d}(\rz^{2n})$, $d\in\nz$, satisfy condition $(\mathrm{E}')$. 
Let $q\in\Gamma^\omega(\rz^{2n})$, $\omega\in\rz$, and 
let $N$ be a positive integer so large that $\omega-dN<-2n$. Then 
 $$\mathrm{Tr}\,\big( q(x,D)(\lambda-p_0(x,D))^{-N}\big)\sim_{\lambda\to+\infty} 
      \sum_{j=0}^{+\infty}c_j\lambda^{\frac{\omega+2n-j}{d}-N}+\sum_{\ell=0}^{+\infty}
    \big(c_\ell^\prime\log\lambda+c_\ell^{\prime\prime}\big)\lambda^{-N-\ell}$$
for suitable coefficients depending on $p_0$ and $q$. 
\end{theorem}
\begin{proof}
Let $b\in\bfGa^{-d,d}(\rz^{2n}\times\rpbar)\subset \wt\bfGa^{-d,0}(\rz^{2n}\times\rpbar)$ be the parametrix from Theorem \ref{thm:resolvent}. 
Then $a:=q\#b^{\# N}\in \wt\bfGa^{\omega-dN,\omega}(\rz^{2n}\times\rpbar)$ and, due to Theorem  \ref{thm:trace-expansion-01} 
$($with $d$ replaced by $\omega-dN$, $\nu$ replaced by $\omega$, and $m=2n)$, we have 
\begin{equation}\label{eq:trace}
 \mathrm{Tr}\,a(x,D,\mu)\sim_{\mu\to+\infty} \sum_{j=0}^{+\infty}c_j\mu^{\omega-dN+2n-j}+\sum_{\ell=0}^{+\infty}
    \big(c_\ell^\prime\log\mu+c_\ell^{\prime\prime}\big)\mu^{-dN-\ell}.
\end{equation}
Since $a(x,D,\mu)=q(x,D)(\mu^d-p_0(x,D))^{-N}$, the result follows simply by substituting $\mu=\lambda^{1/d}$ provided we can show that the coefficients $c_\ell^\prime$ and $c_\ell^{\prime\prime}$ in \eqref{eq:trace} can be non-zero only if $\ell$ is an integer multiple of $d$. Inspecting the proof of Theorem \ref{thm:trace-expansion-01}, this will be the case if we can show that the symbols $r_{\{\omega-dN-J,\omega-J+\ell\}}$ in \eqref{eq:r-expansion} and the symbols $q_{j,\ell}$ in \eqref{eq:for-later} can be non-zero only if $\ell$ is a multiple of $d$. In this regard, since $q(x,\xi)$ does not depend on $\mu$, it is no restriction of generality to assume that $\omega=0$ and $q=1$, i.e., $a=b^{\# N}$ is the $N$-fold Leibniz product of the parametrix of 
$p(x,\xi,\mu)=\mu^d-p_0(x,\xi)\in \wt\bfGa^{d,0}(\rz^{2n}\times\rpbar)$.  

Obviously, $p^\infty_{\{d,0\}}=1$, $p^\infty_{\{d,d\}}=-p_0$, and $p^\infty_{\{d,j\}}=0$ for all other $j\in\nz$.  Since
 $$\delta_{0\ell}=1_{\{0,\ell\}}=(p\#b)_{\{0,\ell\}}=\sum_{j+k=\ell} b^\infty_{\{-d,j\}}\#p^\infty_{\{d,k\}}\qquad\forall\;\ell\in\nz_0,$$
cf. Remark \ref{rem:limit-symbols}, it follows that $b^\infty_{\{-d,j\}}=0$ whenever $j$ is not a multiple of $d$ (and $b^\infty_{\{-d,d\ell\}}=(-1)^\ell p_0^{\#\ell}$ for every $\ell\in\nz_0$). By induction, the same is then true for $a=b^{\#N}$. 

Any homogeneous component $a^{(-Nd-j,-j)}(x,\xi,\mu)$ is a linear combination of terms of the form $\gamma(x,\xi)(\mu^d-p_0^{(d)}(x,\xi))^{-M}$ with $\gamma\in\Gamma^{(M-N)d-j}_\Hom(\rz^{2n})$ and a positive integer $M$. In fact, for $N=1$, i.e., $a=b$, 
this follows from 
 $$\delta_{0\ell}=(b\#p)^{(-j,-j)}=\sum_{k+\ell+2|\alpha|=j}\frac{1}{\alpha!}
    (\partial^\alpha_\xi b^{(-d-k,-k)})(D^\alpha_x p^{(d-\ell,-\ell)})$$
which yields $b^{(-d,0)}=(\mu^d-p_0^{(d)})^{-1}$ and the recursion formula 
 $$b^{(-d-j,-j)}=-\sum_{\substack{k+\ell+2|\alpha|=j\\ k\le j-1}}\frac{1}{\alpha!}
     (\partial^\alpha_\xi b^{(-d-k,-k)})(D^\alpha_x p_0^{(d-\ell)})(\mu^d-p_0^{(d)})^{-1}$$
whenever $j\ge1$. For general $N$ proceed by induction, using the formula for the homogeneous components of 
the composition $a=b\#b^{\#(N-1)}$. 
Moreover, 
%It remains to observe that
% $$(\mu^d-p_0^{(d)})^{-1}-\sum_{k=0}^{K-1} (p_0^{(d)})^k\mu^{-(k+1)d}=\mu^{-Kd}(p_0^{(d)})^K(\mu^d-p_0^{(d)})^{-1}=:r_K,$$
%where $r_k$ is homogeneous of degree $-d$ on $(\rz^{2n}\setminus\{0\})\times\rz_+$ and 
%$|r_K(x,\xi,\mu)|\lesssim \mu^{-(K+1)d} |(x,\xi)|^{Kd}$. 
 $$(\mu^d-p_0^{(d)})^{-M}=\sum_{\ell=0}^{L-1} (p_0^{(d)})^{\ell}\mu^{-(M+\ell)d}+r_{ML}$$
for arbitrary $L$, where $r_{ML}$ is homogeneous of degree $-Md$ on $(\rz^{2n}\setminus\{0\})\times\rz_+$ and 
$|r_{ML}(x,\xi,\mu)|\lesssim \mu^{-(M+L)d} |(x,\xi)|^{Ld}$. For $M=1$ this follows from the (finite) geometric sum formula, 
for general $M$ by induction.  This completes the proof. 
\end{proof}

%%%%%%%%%%%%%%%%%%%%%%%%%%%%%%%%%%%%%%%%
%%%%%%%%%%%%%%%%%%%%%%%%%%%%%%%%%%%%%%%%
\bibliographystyle{amsalpha}

%%%%%%%%%%%%%%%%%%%%%%%%%%%%%%%%%%%%%%%%%%%%%%%%%%%%
%%%%%%%%%%%%%%%%%%%%%%%%%%%%%%%%%%%%%%%%%%%%%%%%%%%%
\end{document}